\documentclass[a4paper,11pt]{article}
\usepackage[utf8]{inputenc}
\usepackage{amssymb}
\usepackage{amsfonts}
\usepackage{amsmath}
\usepackage{amsthm}
\usepackage[colorlinks=true]{hyperref}

\usepackage{url}
\def\dbar{\bar\partial}

\def\Tr{\mathrm{tr}}
\def\Vol{\mathrm{Vol}}

\def\C{\mathrm{C}}
\def\dbar{\bar\partial}

\def\Tr{\mathrm{tr}}
\def\Vol{\mathrm{Vol}}

\def\Id{\mathrm{Id}}
\def\Scal{\mathrm{scal}}

\def\PP{\mathbb{P}}

\newtheorem{lemma}{Lemma}[section]
\newtheorem{corollary}{Corollary}[section]
\newtheorem{proposition}{Proposition}[section]

\newtheorem{theorem}{Theorem}[section]
\theoremstyle{definition}
\newtheorem{definition}[theorem]{Definition}

\title{About projectivisation of Mumford semistable bundles over a curve}
\author{J. Keller}
\begin{document}

\maketitle

\begin{abstract}
We study the existence of canonical K\"ahler metrics on the projectivisation of strictly Mumford semistable vector bundles over a curve.  
\end{abstract}

Consider the projectivisation $\PP(E)$ of a holomorphic vector bundle $E$ over a smooth base manifold $B$ polarized by the ample line bundle $L_B$. Various results have established a relationship between the stability of the underlying bundle $E$ and the existence of K\"ahler metrics with special curvature on $\PP(E)$, at least when $c_1(L_B)$ can be endowed with an extremal metric. The case of a base manifold of complex dimension 1 has been intensively studied. Building on the work of D. Burns- P. de Bartolomeis, E.\ Calabi, A.\ Fujiki, C.\ Lebrun and many others, V.\ Apostolov, D.\ Calderbank, P.\ Gauduchon and C.\ T\o nnesen-Friedman have provided a complete understanding of the situation for stable or polystable bundles over a smooth Riemann surface. They showed that there is a K\"ahler metric with constant scalar curvature (cscK metric in short) in any K\"ahler class on $\PP(E)$ if and only if the bundle $E$ is Mumford polystable \cite{ACGT,ACGT2,AT1}. Another approach was also carried out in a series of paper of Y.-J.\ Hong, see \cite{H} who investigated the case of higher dimensional base. Other results for ruled manifolds appeared recently in relationship with extremal K\"ahler metrics in \cite{Bro,LS}. Up to our knowledge, the case of strictly semistable bundle is still open in complete generality. We expect that for a base manifold of dimension $\geq 2$ all the phenomena of stability for $\PP(E)$ could happen when $E$ is Mumford semistable (see for instance \cite{KR1}).  \\
In this note, we essentially study the particular case of a ruled surface given by the projectivisation of a Mumford semistable vector bundle (which is not stable) over a Riemann surface of genus $g\geq 2$. Some partial generalizations are given for higher dimensional base (note that the results of Section \ref{cscK conic} will be extended in a forthcoming paper).

\noindent {\bf Conventions}: If $\pi\colon E\to B$ is a vector bundle then $\pi\colon \PP(E)\to B$ shall denote the space of complex \emph{hyperplanes} in the fibres of $E$.  Thus $\pi_* \mathcal O_{\mathbb P(E)}(r) = S^r E$ for $r\ge 0$.

\section{Almost constant scalar curvature metric and K-semistability}
\subsection{Definition of an almost cscK metric}
Let $X$ be a polarized manifold by an ample line bundle $L$. In \cite{D4}, S.K. Donaldson introduced the notion of a test configuration for the couple $(X,L)$. Essentially, it consists in a $\mathbb{C}^*$-equivariant flat projective family $\pi:(\mathcal{X},\mathcal{L})\rightarrow \mathbb{C}$ with generic fibre $\pi^{-1}(t)$ isomorphic to $(X,L)$ for $t\neq 0$ while the central fibre, that may be singular, has a $\mathbb{C}^*$ action. Using this action, one can define a numerical invariant $F_1(\mathcal{X},\mathcal{L})$, which generalizes the Futaki invariant. Then $(X,L)$ is said to be K-stable if $F_1(\mathcal{X},\mathcal{L})>0$ for any non trivial test-configurations with general fibre $(X,L)$. We refer to \cite{RT, Thnotes} for details about this definition and to \cite{Stoppa2} for the notion of non trivial test-configuration. The Yau-Tian-Donaldson conjecture asserts that the existence of a cscK in the class $c_1(L)$ is equivalent to the K-stability of $(X,L)$.\\
It is natural to ask what is happening at the limit case for K-semistability. \\
Let us denote $\hat{s}_L$ the average of the scalar curvature in the class $c_1(L)$, which is a topological invariant. For $\omega$ a K\"ahler metric in $c_1(L)$, we denote $\Scal(\omega)$ its scalar curvature. From the main result of \cite{D5}, we know that if there exists a sequence of metric $\omega_\epsilon$ such that $$\Vert\Scal(\omega_\epsilon)-\hat{s}_L\Vert_{L^2}\rightarrow 0,$$ then the manifold is K-semistable. The converse is an open question as far as we know.

\begin{definition}
 Given $(X,L)$ a projective manifold, we say that there exists an almost cscK metric in the class $c_1(L)$ in $\C^r$ topology $(r\in \mathbb{N})$ if there is a family of K\"ahler metrics $\omega_\epsilon\in c_1(L)$ such that $$\Vert\Scal(\omega_\epsilon)-\hat{s}_L\Vert_{\C^r}\rightarrow 0$$ when $\epsilon\rightarrow 0$. 
\end{definition}
In the case of the anticanonical class, this definition appeared first in \cite{Bando} where it is related to the existence of a lower bound for the Mabuchi K-energy. Obviously, from Donaldson's result, a manifold $(X,L)$ endowed with an almost cscK metric is K-semistable.

\subsection{Construction of almost cscK metric}\label{Sect1}
Let $E$ be an irreducible Mumford semistable vector bundle of rank 2 over a polarized manifold $(B,L_B)$, given by a non-split exact sequence of line bundles 
$$0\rightarrow L_1 \rightarrow E \rightarrow L_2 \rightarrow 0,$$
with $c_1(L_1)=c_1(L_2)$. Let us assume that $h_1,h_2$ are projectively flat metrics on $L_1,L_2$ satisfying $F_{h_1}=\deg_L(L_1)\omega=F_{h_2}$ with $\omega$ a cscK metric in the class $c_1(L_B)$.
Consider the holomorphic structure on $E$ that has the following form 
$$\bar{\partial}_E=\begin{pmatrix}
                    \bar{\partial}_{L_1} & \alpha \\
                    0  & \bar{\partial}_{L_2}
                   \end{pmatrix}
$$
where $\alpha$ is a smooth section of $\Omega^{0,1}(Hom(L_1,L_2))$, see \cite[Chapter I, Section 6]{Kob}. Then
one has for the curvature of $E$, and denoting $\mu(E)$ the slope of $E$,
\begin{align*}
\Vert  F_{E}- \mu(E)\Id_E\omega\Vert_{\C^r} \leq&\Vert  F_{h_1}- \deg(L_1)\omega \Vert_{\C^r} + \Vert  F_{h_2}- \deg(L_2)\omega\Vert_{\C^r}\\
& + 2\Vert \alpha\Vert_{\C^r}^2 + 2\Vert \dbar^* \alpha\Vert_{\C^r}^2
\end{align*}
We can do a gauge change of the form $g=\begin{pmatrix} \xi & 0 \\ 0 & \xi^{-1} \end{pmatrix}$ and we obtain  $g(\bar{\partial}_E)=\begin{pmatrix}
                    \bar{\partial}_{L_1} & \xi^{-2}\alpha \\
                    0  & \bar{\partial}_{L_2}
                   \end{pmatrix}.$
For any any $\epsilon>0$ and for any $r>0$, we can find the gauge transformation $\xi$ such that $$2\xi^{-2}(\Vert \alpha\Vert_{\C^r}^2 + \Vert \dbar^* \alpha\Vert_{\C^r}^2)< \epsilon.$$ This provides a structure $h_E$ (depending on the parameters $\epsilon,r$) such that $$ \Vert F_{E,h_E}-\mu(E)\Id_E\omega\Vert_{\C^r}< \epsilon.$$ Note that fixing the holomorphic structure with variation of the metric, or fixing the metric with variation of the holomorphic structure is geometrically equivalent in this setup. Therefore, we have obtained an approximate Hermitian-Einstein structure in the sense of \cite[Chapter IV]{Kob}, from which  we deduce an almost cscK metric on $\PP(E)$ using the next lemma.\\
From now we assume that $E$ is ample, that is $\mathcal{O}_{\PP(E)}(1)$ is a positive line bundle (without loss of generality we can tensorize $E$ by a sufficiently ample line bundle $L_C$, use the identification $\PP(E\otimes L_C)\simeq \PP(E)$ and the induced approximate Hermitian structure). 
\begin{lemma}\label{tech}
 From $h_E$ hermitian metric on the bundle $E$, one can define a metric $\hat{h}_E$ on $\mathcal{O}_{\PP(E)}(1)$ which curvature is denoted $\hat{\omega}_E$ and is a K\"ahler form. Then at  $v\in \PP(E)$, with $\pi(v)=x\in C$, one has pointwise
$$\hat{\omega}_E = \pi^* \left(\frac{\sqrt{-1}}{\Vert v\Vert^2_{h_E}}\langle F_{E,h_E}(v),v\rangle_{h_E}\right)+{\omega_{FS}}_{\vert \PP(E)_x}$$
and ${\omega_{FS}}_{\vert \PP(E)_x}$ is the Fubini-Study metric at $\PP(E)_x$. 
\end{lemma}
We refer to \cite[Chapter V, §15.C]{De2} for a proof. A direct consequence of the previous lemma and the existence of an approximate Hermitian-Einstein structure with respect to a cscK metric is the following proposition.
\begin{proposition}\label{prop0}
 Let $E\rightarrow B$ be an ample Mumford semistable rank 2 vector bundle induced by a non-split exact sequence of projectively flat line bundles as above over a cscK polarized manifold $(B,L)$.  Then for any $r\in \mathbb{N}$, there is an almost cscK metric on the ruled surface  $\pi\colon \PP(E)\to B$ in $\C^r$ topology.
\end{proposition}

If the base manifold is a curve of genus $g>1$, line bundles are automatically projectively flat, the exact sequence does not split for $L_1$ not isomorphic to $L_2$ since $h^1(C,L_2\otimes L_1^*)=g-1>0$ and there exists a cscK metric on the base manifold. Thus we obtain the next result.
\begin{corollary}\label{cor1}
 Consider $E$ a rank 2 vector bundle on a curve $C$ of genus $\geq 2$. Assume that $E$ is Mumford semistable.  Then for any $r\in \mathbb{N}$, there is an almost cscK metric on the ruled surface  $\pi\colon \PP(E)\to C$ in $\C^r$ topology.
\end{corollary}
Note that if $E$ is not irreducible, then it is actually a direct sum of line bundles and thus we know the existence of a genuine cscK metric on the projectivisation, see for instance \cite{AT1}.

\subsection{Computation of the Donaldson-Futaki invariant}

\begin{proposition}\label{prop}
Consider $E$ an ample irreducible Mumford semistable vector bundle which is not stable over a curve of genus $g>1$. Then $(\PP(E),\mathcal{O}_{\PP(E)}(1))$ is not K-polystable and not asymptotically Chow polystable.  
\end{proposition}
\begin{proof}
This is a consequence of \cite[Theorem 5.13]{RT}, where it is done a computation of the Donaldson-Futaki invariant for the test configuration induced by a deformation to the normal cone of $\PP(F)$ where $F$ is any subbundle of $E$. This computation  shows that the Donaldson-Futaki invariant for such a test-configuration is a multiple of the differences of slopes $\mu(E)-\mu(F)$. Remark that with \cite[Proposition 4.1, Theorem 4.5]{DV1}, it is also proved that $\mathcal{O}_{\PP(E)}(1)$ is not asymptotically Chow polystable.  
\end{proof}

\begin{proposition}\label{prop3}
 Assume that  $E\rightarrow B$ an ample Mumford semistable rank 2 vector bundle induced by a non-split exact sequence of projectively flat line bundles over a cscK polarized manifold $(B,L)$, as constructed in Section \ref{Sect1}. Then $(\PP(E),\mathcal{O}_{\PP(E)}(1))$ is not K-polystable and not asymptotically Chow polystable.  
\end{proposition}
\begin{proof}
Let us denote $b=\dim_\mathbb{C} B$. We compute the Donaldson-Futaki invariant $F_1$ for the test configuration induced by a deformation to the normal cone of $\PP(L_1)$. Note that $\mu(L_1)=\mu(E)$. As explained in \cite{RT} (see also \cite{KR1}), $F_1=a_1b_0-a_0b_1$
where one has defined
\begin{align*}
p(r)&=h^0(\PP(E),\mathcal{O}_{\PP(E)}(r))=a_0r^{b+1}+a_1r^b+...\\
w(r)&=\sum_{i=0}^r ih^0(B,L_1^i\otimes L_2^{r-i})=b_0r^{b+2}+b_1r^{b+1}+...
\end{align*} 
Now, under our assumptions, using the fact that $c_1(E)=2c_1(L_1)$, the polynomials $p$ and $w$ are proportionals. This shows that $F_1$ vanishes and a similar reasoning can be done to get that the Chow weight associated to this test configuration also vanishes. 
\end{proof}

\begin{corollary}\label{cor}
 There exist examples $(X,L)$ of polarized manifolds such that $L$ is K-semistable and not K-polystable. In particular, there are examples of non convergent sequence of almost cscK metrics. For any irreducible Mumford semistable bundle $E$ (not Mumford stable) of rank 2 over a curve of genus $g\geq 2$, there are integral classes on the ruled surface $X=\PP(E)$ that are K-semistable and not K-polystable.  
\end{corollary}

\begin{proof}
We fix a bundle as in the statement and apply Corollary \ref{cor1} to produce a sequence of an almost cscK metric. The existence of such metric implies in turn that $(\PP(E),\mathcal{O}_{\PP(E)}(1)))$ is K-semistable from \cite{D5}. On another hand, the automorphism group  $Aut(\PP(E))$ is actually trivial, see \cite{S}. Therefore if the sequence of almost cscK metric was convergent, it would converge towards a cscK metric and $(X,L)$ would be K-stable by  \cite{D1,Stoppa1}. This would contradict Proposition \ref{prop}. 
\end{proof}

Corollary \ref{cor} gives a positive answer to a conjecture of J. Stoppa \cite{Sto3}. Our construction can be easily modified to produce examples of K-semistable not K-stable manifolds in any dimension using vector bundles of higher rank over a curve (using an induction argument on their Harder-Narasimhan filtration, see \cite{Ke00}) or using Proposition \ref{prop3}. Note that Fano examples of K-semistable but not K-polystable threefolds have been found by G. Tian by considering small deformations of the Mukai-Umemura threefold. 

\section{Almost balanced metric and Asymptotic Chow semistability}

\subsection{Definition of almost balanced metric}
Let us recall that one can define for $X$ a submanifold of $\mathbb{P}^N$ the center of mass of $X$
as 
$$\mu(X)=\int_X \frac{zz^*}{\vert z\vert^2} d\mu_{FS} - \frac{\Vol(X)}{N+1}\Id\in \sqrt{-1}Lie(SU(N+1))$$
considering $\mathbb{P}^N$ as a co-adjoint orbit in the Lie algebra of $SU(N+1)$. The Chow weight of $X$ with respect to $A$, hermitian matrix, 
is $$FCh(A,X)=\Tr(\mu(X)\cdot A)=\int_X \frac{z^* A z}{\vert z \vert^2}d\mu_{FS} - \frac{\Vol(X)}{N+1}\Tr(A)$$
and the definition can be extended to any algebraic cycles. It is a classical fact, based on Kempf-Ness theory, that $FCh(A,e^{tA}\cdot X)$ is an increasing function of $t$ and we refer \cite[Proposition 5]{D5}  for details. In particular for $\overline{X}$ the limiting Chow cycle of $e^{tA}\cdot X$ as $t\rightarrow -\infty$, we get $$FCh(A,X)\geq FCh(A,\overline{X}).$$ This provides the inequality
\begin{equation}\Vert \mu(V)\Vert_{2} \,\Vert A\Vert_2 \geq -FCh(A,\overline{X}) \label{ineq3} 
\end{equation}
where one has defined the norm $\Vert T\Vert_2^2= \sum \vert \lambda_i\vert^2$ for $\lambda_i$ eigenvalues of the hermitian matrix $T$, taking into account their multiplicities. This is the finite dimensional analogue of the main theorem of \cite{D5} that we used previously.
Let us now introduce a notion of almost-balanced metrics. 
\begin{definition}\label{albal}
 Given $(X,L)$ a projective manifold, we say that there exists a sequence of almost balanced metrics if for all $k>>0$ and all $\epsilon>0$,
there exists a hermitian metric $h_{k,\epsilon}$ on $L^k $ such that the  Bergman function satisfies $$\Big\Vert\rho(h_{k,\epsilon}) - \frac{N_k+1}{\Vol_L(X)}\Big\Vert_{\C^0}\leq \epsilon.$$ 
\end{definition}
Note that $\rho(h_{k,\epsilon})=\sum_{i=1}^{N_k+1} \vert s_i\vert^2_{h_{k,\epsilon}}\in C^{\infty}(X,\mathbb{R}_+)$ for $N_k+1=h^0(L^k)$ and $\{s_i\}_{i=1,..,N_k+1}$ an orthonormal basis of $H^0(L^k)$ with respect to the $L^2$-inner product induced by $h_{k,\epsilon}$. The existence of an almost balanced metric for $(X,L)$ implies, using \eqref{ineq3} and Hilbert-Mumford criterion, that $(X,L)$ is asymptotically Chow stable since the Chow weight of the limiting Chow cycle along a test-configuration cannot be strictly negative (this appears also clearly in \cite[Equation (16)]{D5} where the lower order terms are the higher Chow weights associated to the one-parameter subgroup action).

\subsection{Construction of almost balanced metric}

Let us consider $\hat{h}_\epsilon$ a hermitian metric on $\mathcal{O}_{\PP(E)}(1)$ with the same notations as in the previous section.
Then the Bergman function for $(\mathcal{O}_{\PP(E)}(k),\hat{h}_\epsilon^k)$ has an asymptotic expansion
\begin{equation}\rho(\hat{h}_\epsilon^k)=k^{r}+ k^{r-1}\frac{\Scal(\omega_{\epsilon})}{2}+k^{r-2}a_2+...+k^{r-q}a_q\label{asympt}\end{equation}
where $r$ is the rank of the bundle $E\rightarrow C$ and $\omega_{\epsilon}=c_1(\hat{h}_\epsilon)$.
The writing of \eqref{asympt} means the following inequality holds in $\C^0$-topology (it will be sufficient to work in that topology in the sequel)
\begin{equation}\Big\Vert \rho(\hat{h}_\epsilon^k)- \left( k^{r}+ k^{r-1}\frac{\Scal(\omega_{\epsilon})}{2}+k^{r-2}a_2+...+k^{r-q}a_q\right) \Big\Vert_{\C^0}\leq C_q(\hat{h}_\epsilon)k^{n-q-1}.\label{asympt2}\end{equation} 
The terms $a_i$ involve at most the $(2i-2)$-th first covariant derivatives of the curvature $\omega_{\epsilon}$. 

\begin{lemma}\label{fixed}
 If the metric $\omega_\epsilon$ is bounded from below and bounded in $\C^{2q}$ norm by a constant $\delta$ with some reference metric, then 
the constant $C_q(\hat{h}_\epsilon)$ in Equation \eqref{asympt2} depends actually only on $q$ and the constant $\delta$. 
\end{lemma}
This is well known, see for instance \cite[Proposition 6]{D1}. \\

We are now coming back to the setup of Section \ref{Sect1} and shall construct a sequence of almost-balanced metric. Since we work in the smooth category, it is not difficult to adapt the reasoning in order to obtain an approximate Hermitian-Einstein structure $h_{\epsilon}^\infty$ in the following sense
\begin{equation}\label{uniform2}\Vert  F_{E,h_{\epsilon}^\infty}-\mu(E)\Id_E\,\omega\Vert_{\C^\infty}< \epsilon.\end{equation}
Furthermore we can assume that $h_E$ is real analytic. If $h_E$ is not real analytic we may use a slight generalization of Tian's result \cite{Ti1} of approximation of a positive hermitian metric by a sequence of Bergman type metrics in smooth topology (for a discussion on the smooth convergence see \cite{Ru}). Actually, we can pull-back the canonical metric on the universal bundle $U_{(2)} $ over the Grassmannian $Gr(2,H^0(B,E\otimes L^s))$ for $s>>1$, which provides a sequence of real analytic metrics. This sequence is convergent towards the metric $h_E$ in smooth topology thanks to the asymptotic result for the Bergman kernel of $E\otimes L^s$ that can be found in \cite{W2}. 
\\
On $X=\PP(E)$, the curvature $\omega^\infty_\epsilon$ of the associated real analytic metric $\hat{h}_\epsilon^\infty$ on $\mathcal{O}_{\PP(E)}(1)$ is bounded in $\C^\infty$ norm and is positive. This can be seen by expressing the curvature of $\hat{h}_\epsilon^\infty$ in terms of the curvature terms of $h_\epsilon^\infty$, see Lemma \ref{tech}. We can apply Lemma \ref{fixed}  and for $0<\epsilon<\epsilon_0$ we get a uniform expansion of the Bergman function of $\hat{h}_\epsilon^\infty$ with constant depending only on the maximum order of the expansion and $\epsilon_0$, and hence we denote the constant in Equation \eqref{asympt2} by $C_q(\epsilon_0)$. 
\begin{lemma}\label{uniform3}
In the above setting, there is a constant $C_\infty>0$ depending only on $\epsilon_0$ such that $C_q(\epsilon_0)\leq C_\infty^q$. In other words, the growth of the error constant in \eqref{asympt2} when taking higher order expansion is at most exponential.
\end{lemma}
\begin{proof}
This is a consequence of the techniques used in the proof of \cite[Theorem 1.3]{L-L} (see also Theorem 1.2 of the associated announcement paper). We shall use the notations of the quoted paper. The Bergman function is given by the sum of an orthonormal basis of sections that can be taken as the union of a peaked section and vanishing sections at $x$ that have been orthonormalized. In order to orthonomalize these sections, it is necessary to inverse a matrix formed of the inner products $\langle S,S\rangle_{L^2}$ and $\langle S,T\rangle_{L^2}$ as it is done in \cite[Section 5]{L-L}. 
Since we deal with analytic metrics, we can do the Taylor expansion of the involved metrics (on the line bundle and the volume form) and we get an expansion of the $L^2$ norm $\langle S,S\rangle_{L^2}$ of the peak section $S$ of the form $\langle S,S\rangle_{L^2}=\frac{1}{k^n}(1+\sum_{i\geq 1} \beta_i k^{-i})$. By convergence of this Taylor expansion, one has for a certain uniform constant $C>0$ that depends on the metric, $\vert \beta_i\vert \leq C^i$ for $i\geq 1$. In \cite[Theorem 4.1 (3)]{L-L}, it is proved a uniform bound on the $L^2$-inner product $\langle S,T\rangle_{L^2}$ between a holomorphic section $S$ peaked at the point $x$ and sections $T$ that vanish at order $p'>0$ at $x$. Together these two uniform controls provide the expected growth on the error term $C_q$ of the expansion of the Bergman function. In particular we have shown that in \eqref{asympt2}, one has the control $\vert a_i \vert <\frac{c_0}{\gamma^i}$ for a uniform constant $\gamma>0$. 
\end{proof}

Furthermore, the terms $a_i$ enjoy the property of being polynomial expressions in the curvature and its covariant derivatives. For any integer $q$, we can find a metric $\hat{h}_{\epsilon,q}^\infty$ such that the term
$$\left( k^{r}+ k^{r-1}\frac{\Scal(\omega_{\epsilon})}{2}+k^{r-2}a_2+...+k^{r-q}a_q\right),$$
which satisfies a similar property, is also constant up to an error term of the form $\epsilon/2$ in $\C^0$ norm while we are are still under the assumptions of Lemma \ref{fixed}. This is a consequence of the uniform control in $\C^\infty$ topology of the curvature of $\omega^\infty_\epsilon$ using \eqref{uniform2}. \\
Furthermore, using Lemma \ref{uniform3} and taking $k>C_\infty$, we can impose $$k^{r-q-1}C_{q}(\epsilon_0)\leq k^{r-1}\left(\frac{C'_\infty}{k}\right)^q<\epsilon/2$$ in Equation \eqref{asympt2} by choosing $q$ large enough. Hence, for $k>>0$, we have obtained a metric $\hat{h}_{k,\epsilon}=(\hat{h}_{\epsilon,q}^\infty)^k $ which is almost balanced in the sense of Definition \ref{albal}.\\
Eventually, with the examples considered in Section \ref{Sect1} and Proposition \ref{prop}, we have proved the following result.
\begin{theorem}\label{thm2}
 There exist examples $(X,L)$ of smooth polarized manifolds such that $L$ is asymptotically Chow semistable and not K-polystable.
 For instance, for any irreducible Mumford semistable bundle $E$ (not Mumford stable) of rank 2 over a curve of genus $g\geq 2$, any K\"ahler integral class on the ruled surface $X=\PP(E)$ is asymptotically Chow semistable, not asymptotically Chow polystable and not K-polystable.  
\end{theorem}
\medskip

Note that Chow semistability implies K-semistability, see \cite{Thnotes}. Therefore Theorem \ref{thm2} implies straightforward Corollary \ref{cor}.
We chose to present both results in order to stress how one could expect to find examples of K-semistable but non Chow semistable manifolds. Actually, one could try to find a sequence of almost cscK metrics in $\C^0$ norm such that some high order covariant derivatives of the Riemann curvature tensor are unbounded.

\section{Parabolic structures and conic K\"ahler metrics \label{cscK conic}}

\subsection{Construction of stable parabolic structure}

Given a vector bundle $E$ over a curve, we provide an elementary construction that allows to construct a stable parabolic structure on $E$. 

\subsubsection{From an unstable to a semistable parabolic structure\label{sub1}}
Suppose that $E$ has rank 2 and is not semistable.
We have an exact sequence
$$
 0\longrightarrow F_1\longrightarrow E\longrightarrow
 F_2\longrightarrow 0
$$
such that\\
(i) $F_i$ are line bundles,\\
(ii) $\deg(F_1)>\mu(E)=\frac{\deg(F_1)+\deg(F_2)}{2}>\deg(F_2)$.\\

We set
\begin{equation}\label{A}
 A:=\deg(F_1)-\mu(E)=\mu(E)-\deg(F_2)>0.
\end{equation}
We take a large integer $N$ such that
$2A/N<1$ and fix some distinct points $p_i\in X$ for $i=1,\ldots,N$.
We take subspaces
$V_i$ of $E_{|p_i}$ $(i=1,\ldots,N)$ of dimension 1 such that
$V_i\not\subset F_{1|p_i}$.\\
We define the parabolic filtration $\mathcal{F}_{j}(p_i)$ at $p_i$ by $$\mathcal{F}_{0}=E_{|p_i}\supset \mathcal{F}_{1}=V_i \supset \mathcal{F}_{2}=\{0\}$$
with weight $0$ for $\mathcal{F}_0$ and $2A/N$ for $\mathcal{F}_1$. The, the degree of the parabolic bundle $E_{\ast}$ is
$$
 \mathrm{par}\deg(E_{\ast})=\deg(E)+\sum_{i=1}^N 2A/N=\deg(E)+2A$$
Hence, $\mathrm{par}\mu(E_{\ast})=\mu(E)+A$.
The degree of the induced parabolic bundle $F_{1\ast}$ is
$$
 \mathrm{par}\deg(F_{1\ast})=\deg(F_1)=\mu(E)+A=\mathrm{par}\mu(E_{\ast})$$
If $F\subset E$ is a subsheaf of rank one such that $F\not\subset F_1$,
then there exists a non-trivial morphism $F\longrightarrow F_2$.
Hence, we have $\deg(F)\leq \deg(F_2)$ and thus, using (\ref{A}),
$$
 \mathrm{par}\mu(F_{\ast})=\mathrm{par}\deg(F_{\ast})
\leq
 \deg(F)+2A
\leq
 \mu(E)+A
=\mathrm{par}\mu(E_{\ast})
$$
Hence, we have proved that $E_{\ast}$ is parabolic semistable.

\subsubsection{From a semistable to a stable parabolic structure} 

Let $E_{\ast}$ be semistable (not stable) parabolic bundle of rank $2$ over
 $(X,D)$, where $D$ is a finite subset of $X$. We wish to modify the parabolic structure
so that the new parabolic bundle is stable. One of the following holds:
\begin{enumerate}
\item $E_{\ast}\simeq F_{0\ast}\otimes \mathbb{C}^2$,
\item $E_{\ast}\simeq F_{1\ast}\oplus F_{2\ast}$,
\item there exists a non-split exact sequence
 $$0\longrightarrow F_{1\ast}\longrightarrow
 E_{\ast}\longrightarrow F_{2\ast}\longrightarrow 0,$$
 where $F_{i\ast}$ are parabolic bundles of rank one.
\end{enumerate}

We choose $p'_1,p'_2,p'_3\in X\setminus D$
and subspaces $V_i\subset E_{|p'_i}$ of dimension 1
with the following property: if  $F_{\ast}\subset E_{\ast}$
with $\mathrm{par}\deg(F_{\ast})=\mathrm{par}\mu(E_{\ast})$, then $F_{|p'_i}=V_i$ may happen
for at most one $i$.\\
Let us choose a real number $1>\epsilon>0$.
We consider the parabolic filtration $\mathcal{F}'_{j}(p'_i)$ at $p'_i$ by $$\mathcal{F}'_{0}=E_{|p'_i}\supset \mathcal{F}'_{1}=V_i \supset \mathcal{F}'_{2}=\{0\}$$
with weight $0$ for $\mathcal{F}'_0$ and weight $\epsilon$ for $\mathcal{F}'_1$.
We also consider the same parabolic structure
at the points of $D$.
We obtain a parabolic bundle $E^{(\epsilon)}_{\ast}$. We have
$$
 \mathrm{par}\mu(E^{(\epsilon)}_{\ast})
=\mathrm{par}\mu(E_{\ast})+3\epsilon/2.
$$
For $F_{\ast}\subset E_{\ast}$
such that $\mathrm{par}\deg(F_{\ast})=\mathrm{par}\mu(F_{\ast})=\mathrm{par}\mu(E_{\ast})$,
we have
$$
 \mathrm{par}\deg(F^{(\epsilon)}_{\ast})
\leq
 \mathrm{par}\deg(F_{\ast})+\epsilon
<\mathrm{par}\mu(E^{(\epsilon)}_{\ast})
$$
If $\epsilon$ is sufficiently small,
then
$\mathrm{par}\mu(F^{(\epsilon)}_{\ast})
<\mathrm{par}\mu(E^{(\epsilon)}_{\ast})$
for any $F_{\ast}$ such that
$\mathrm{par}\mu(F_{\ast})<\mathrm{par}\mu(E_{\ast})$.
Finally, with the new induced parabolic structure, $E_{\ast}^{(\epsilon)}$ is parabolic stable over the rank 2 bundle $E$. We explain now how to deal with higher rank bundles.

\subsubsection{Higher rank and corollaries} 
By induction on the rank of the bundle, we have a generalization of our reasoning on any vector bundle over a curve. Let $r\geq 2$ be the rank of $E$. Since the degrees of the subbundles of $E$ are bounded from above, let choose $F_2$ the maximal destabilizing subbundle of $E$. Then $F_1=E/F_2$ is a vector bundle and we have the exact sequence $$ 0 \rightarrow F_1 \rightarrow E \rightarrow F_2 \rightarrow 0.$$ 
Let $F$ a subbundle of $E$. If $F\subset F_1$, then by induction, we can find a stable parabolic stable structure on $F_1$ that we denote $(F_{1})_{\ast_1}$ and thus $\mathrm{par}\mu(F_{\ast_{_1}}) < \mathrm{par}\mu((F_{1})_{\ast_1})$. Let us assume that $A=\mathrm{par}\mu((F_1)_{\ast_1})- \mathrm{par}\mu(E_{\ast_1})$ is non negative. It remains to show that one can refine the considered parabolic structure $\ast_1$ so that $\mathrm{par}\mu(F_{\ast_{_2}}) < \mathrm{par}\mu((F_{1})_{\ast_2}) \leq \mathrm{par}\mu(E_{*_2})$ with respect to a new structure $\ast_2$. This is done as in the previous subsection by choosing an adapted filtration so that 
\begin{align*}
 \mathrm{par}\deg(E_{\ast_2}) &=\mathrm{par}\deg(E_{\ast_1}) + rA,\\
 \mathrm{par}\deg((F_1)_{\ast_2}) &=\mathrm{par}\deg((F_1)_{\ast_1}), \\
 \mathrm{par}\deg((F_2)_{\ast_2}) &\leq \mathrm{par}\deg((F_2)_{\ast_1})+rA.
\end{align*}
If $F\not\subset F_1$, then there is a non trivial morphism $F\longrightarrow F_2$ and eventually $\mathrm{par}\mu(F_{\ast_2})\leq \mathrm{par}\mu(E_{*_2})$. We  get the parabolic semistability of $(E_{\ast_2})$. We apply the same reasoning as in the previous subsection to derive the following result.

\begin{theorem}\label{process}
 Given $E$ a holomorphic vector bundle over a curve, one can find sufficiently many points $p_i$ and sufficiently small weights $\beta>0$ such that the associated parabolic structure $E_\ast$ is Mumford parabolic stable. If $E$ is Mumford semistable and has rank 2, it is sufficient to consider 3 points.
\end{theorem}

Theorem \ref{process} restricted to rational weights and the main result of \cite{Ro1} provide a new proof of an old result of C. Lebrun and M. Singer \cite[Theorem 3.11]{LS}. Note that the assumption on the genus is made to kill the non trivial holomorphic vector fields.

\begin{corollary}
Consider a ruled surface $S$ over a curve $C$ of genus $g\geq 2$. The blow up of $S$ at sufficiently many points admits a cscK metric.  
\end{corollary}

Of course, a different proof can be given using the work of C. Arezzo- F. Pacard. The ruled surface is birationnally equivalent to the product $C\times \mathbb{P}^1$, which means that a blow-up of $S$ is also a blow-up of $C\times \mathbb{P}^1$, on which there exists a cscK metric. Then it is sufficient to apply the main result of \cite{AP1}. Our construction has the advantage to be more constructive.

\subsection{Construction of constant scalar curvature K\"ahler metric with conic singularities}

Let's start with $E$ a semistable bundle over a curve $C$. From the previous section (Theorem \ref{process}), we have obtained a stable parabolic structure $E_*$ along the points $\mathcal{P}=\{p_1,...,p_{m_\mathcal{P}}\}$, with $m_\mathcal{P}\geq 3$ and weight less than $\beta_0>0$. We shall see that there is an Hermitian-Yang-Mills metric with respect to a K\"ahler-Einstein metric with conic singularity along the associated points to the parabolic structure $E_*$.   
Firstly, let us recall the notion of K\"ahler metric with conical singularity, focusing on the complex dimensional one case. 
\begin{definition}
Let $C$ be a complex curve. Let $\mathcal{P}=\{p_j\}_{j=1,..,m_\mathcal{P}}\subset C$ be a finite set of points. Let $\beta=(\beta_1,..,\beta_{m_\mathcal{P}})$ with $0<\beta_i\leq 1$ be the cone angles. Given a point $p_i\in \mathcal{P}$, label a local chart $(V_{p_i},z_1)$ centered at $p_i$ as local cone chart. A K\"ahler metric with conical singularity  and cone angle $2\pi \beta_i$ along $p_i$ (in short a conical K\"ahler metric) is a closed positive (1,1) current and a smooth K\"ahler metric on $C\setminus \mathcal{P}$ such that in a local cone chart $V_{p_i}$ its K\"ahler form is quasi-isometric to the cone flat metric 
\begin{equation}\label{local}\omega_{cone}=\frac{\sqrt{-1}}{2}\beta_i^2 \vert z_1\vert^{2(\beta_i-1)} dz_1\wedge d\bar z_1.\end{equation}
\end{definition}

Over a complex curve, the notions of K\"ahler class and a pointwise conformal class are equivalent. We can apply the work of M. Troyanov \cite{Tro} and therefore fixing at any $m_\mathcal{P}$ points, with at least $m_\mathcal{P}\geq 3$, a curve $C$ admits a conical K\"ahler-Einstein metric $\omega_\beta$ along these fixed points for any angle between 0 and 1. Let us consider a model cone metric of the form $$\tilde{\omega}_\beta=\frac{\sqrt{-1}}{c_\beta}\sum_{i=1}^{m_\mathcal{P}} \partial\bar\partial \vert \sigma_i\vert^{2\beta}+\omega$$
where $\omega$ is a smooth K\"ahler form on $C$, $\beta<\beta_0$, the sections $\sigma_i$ vanish exactly at $p_i$ and $c_\beta>0$ is large enough so that $\tilde{\omega}_\beta$ is positive over the curve. A direct computation shows that actually $\tilde{\omega}_\beta$ satisfies the previous definition. Now, the behavior of conical K\"ahler-Einstein metrics are pretty well understood and it has been derived some Laplacian estimates for it. From \cite[Theorem 1]{Bre1} (see also \cite[Theorem A]{CGP}, \cite[Section 5.2]{GP} and \cite[Theorems 1 and 2]{Yao}), we know that the conical K\"ahler-Einstein metric is uniformly equivalent to the model cone metric for small enough angle, i.e there exists a certain constant $\delta>0$ such that
\begin{equation}
\frac{1}{\delta}\tilde{\omega}_\beta < \omega_\beta < \delta\tilde{\omega}_\beta.
 \label{unif}
\end{equation}

We shall explain now how to apply Simpson's theory \cite{Si1} in our context. Firstly, it is well known that a conical K\"ahler metric has finite volume. Moreover there is an exhaustion function $\varphi$ of $C$ such that $\Delta_{\omega_\beta}\varphi$ is bounded, and Sobolev inequality holds with respect to $\omega_\beta$. Both facts are checked in   \cite[Proposition 4.1]{Li1} where the arguments use only the local expression of the metric \eqref{local} and thus are still valid for $\omega_\beta$. 
Moreover, in dimension one, analytic stability and parabolic stability coincide. In \cite[Section 3]{Li1}, it is constructed a metric $K_0$ on $E_*$ such that its curvature satisfies 
$\vert \Lambda_{\tilde{\omega}_\beta} F_{K_0}\vert_{K_0}$ is bounded on $C$ for small enough angle $\beta>0$. Therefore $\vert \Lambda_{{\omega}_\beta} F_{K_0}\vert_{K_0}$ is still bounded using \eqref{unif}. By stability of the parabolic structure $E_*$, Simpson's theorem \cite[Theorem 1]{Si1} imply the existence of a Hermitian-Yang-Mills metric $H_E$ on $E$ satisfying the Hermitian-Einstein equation \begin{equation}F_{H_E}=\frac{\mathrm{tr} F_{H_E}}{r}\Id_E\,\omega_\beta={\textrm{par}\mu(E)}\Id_E\, \omega_\beta\label{HEeqn}
\end{equation}
over $C\setminus \mathcal{P}$, with $r$ the rank of $E$. 

Let us now consider the ruled manifold $X=\mathbb{P}E$ and $\pi:X\rightarrow C$ the projection on the base manifold. We denote $D\subset \mathbb{P}(E)$ the induced divisor from the preimage of $\mathcal{P}$.
The metric $H_E$ on $E$ induce a metric $\hat{H}_E$ on $\mathcal{O}_{X}(1)$. Actually for any point $p\in C$, $a,b\in E_{\vert p}$, and $\gamma\in E_{\vert p}^*$, one can define locally the metric $\hat{H}_E$ by
$\hat{H}_E(\hat u, \hat v)= \frac{\gamma(u)\overline{\gamma(v)}}{\Vert \gamma\Vert^2_{H_E} }$. \\
The curvature of $\hat{H}_E$ on $X$ is denoted $\hat{\omega}_E$. Since its restriction to the fibre is the Fubini-Study, it is non degenerate on $\mathbb{P}E_{\vert p}$ for $p\in C\setminus \mathcal{P}$. From Equation 
\eqref{HEeqn}, it satisfies $$\hat{\omega}_E={\textrm{par}\mu(E)} \pi^*\omega_\beta,$$ see \cite[Section 1]{Fuj}. Consequently, for any $m$ large enough and from the properties of $\omega_\beta$, the metric $\hat{\omega}_E+m\pi^* \omega_\beta$ is a closed positive current, a K\"ahler metric with conical singularities along $D$ and has constant scalar curvature on $X \setminus D$.
\begin{theorem}\label{thm2b}
 Given $E$ a vector bundle over a curve $C$, there exists a smooth divisor $D\subset \mathbb{P}(E)$ and a constant scalar curvature K\"ahler on $\mathbb{P}(E)$ with conical singularity along $D$. If $E$ is semistable and has rank 2, $D$ can be given by the preimage of 3 points of $C$.
\end{theorem}
We address now some remarks. By changing $m$, this theorem provides in particular examples of cscK cone metric in the class of $[\hat{\omega}_{E}+m\pi^*\omega_\beta]$ which are not K\"ahler-Einstein.  In \cite{RoSi}, it is explained in details how to obtain scalar-flat K\"ahler metrics with orbifold singularities when one restricts to the particular case of a ruled surface and the parabolic structure has rational weights, using Mehta-Seshadri theorem. Also, remark that from the point of view of extremal K\"ahler metrics, some conical metrics are constructed in \cite{Lih1} using the formalism developed by Apostolov et al in \cite{ACGT2}, while G. Sz\'ekelyhidi studied the geometric splitting of a ruled surface given by a non stable manifold under the Calabi flow in \cite{Sz1}. From the point of view of stability, we expect that Theorem \ref{thm2b} provides examples of log-K-stable manifolds in the sense of \cite{D8}.

\bigskip

\noindent {\bf Acknowledgments.} We are very grateful to Vestislav Apostolov, Takuro Mochizuki, Yann Rollin, Xiaowei Wang and Kai Zheng for useful conversations.
This research was partially supported by the French Agence Nationale de la Recherche - ANR project EMARKS.
{\small
\bibliography{Ksstable}
}
\bigskip

\noindent 
{\footnotesize
\textsc{I2M, Centre de Math\'ematiques et Informatique, Aix-Marseille University \\ Technop\^ole de Ch\^ateau Gombert, 39 rue F. Joliot-Curie \\ 13453 Marseille Cedex 13,  FRANCE}}\\
 \url{julien.keller@univ-amu.fr}

\end{document}